\documentclass[a4paper,10pt]{article}
\usepackage[T1]{fontenc}
\usepackage{amsmath}
\usepackage{amsfonts}
\usepackage{amstext}
\usepackage{amsthm}
\usepackage[all]{xy}
\usepackage{geometry}
\usepackage{srcltx}
\usepackage{graphics}
\usepackage{graphicx}
\usepackage{epsfig}
\usepackage{subfigure}
\usepackage{slashed}
\usepackage{color}
\usepackage{amssymb}
\usepackage{srcltx}
\usepackage{comment}
\usepackage[hidelinks]{hyperref}

\newtheorem{theorem}{Theorem}[section]
\newtheorem{prop}[theorem]{Proposition}

\newtheorem*{thma}{Main Theorem}

\theoremstyle{definition}
\newtheorem{defn}{Definition}[section]
\newtheorem{lemma}{Lemma}[section]
\newtheorem{rem}{Remark}[section]

\DeclareMathOperator{\ind}{ind}

\DeclareMathOperator{\eind}{ind_E}

\DeclareMathOperator{\codim}{codim}
\DeclareMathOperator{\inter}{int}
\DeclareMathOperator{\XX}{\hat{X}}
\DeclareMathOperator{\Xx}{\check{X}}

\title{Morse cohomology in a Hilbert space via the Conley Index}
\author{Maciej Starostka}

\begin{document}
\date{}
\maketitle

\begin{abstract}
   The main theorem of this paper states that Morse cohomology groups in a Hilbert space are isomorphic to the cohomological Conley Index. It is also shown that calculating the cohomological Conley Index does not require finite dimensional approximations of the vector field. Further directions are discussed.
\end{abstract}


\section{Introduction.}
The aim of this paper is to show that Morse cohomology groups defined for a certain functional in a Hilbert space can be recovered via the Conley Index. This was motivated by the growing number of different Floer cohomology theories in three and four-dimensional topology. A lot of applications come from the fact that some of those  theories are equivalent. As an example, let us mention the Seiberg-Witten-Floer cohomology and the Embedded Contact Floer cohomology. The equivalence between those two was used to find Reeb orbits on contact manifolds. Some stronger versions of the Weinstein conjectures were obtained (cf. \cite{Taubes}). \\
However, some of the Floer theories are still conjectured to be equivalent, e.g. Seiberg - Witten - Floer ($\mbox{HSW}$) cohomology and Monopole - Floer cohomology ($\mbox{HM}$). The former is defined by the Conley Index while the latter one by counting connecting orbits. The idea of using the Conley Index instead of Floer theory for Seiberg - Witten equations was first introduced by C.Manolescu in \cite{Man}. One of the motivations was the fact that we do not have to deal with transversality. \\
Our approach to the Floer theory via the Conley Index is slightly different than in \cite{Man}. We would like to work with an index pair
 in a Hilbert space and apply the concept of G\k{e}ba-Granas cohomology (see \cite{GG}). This allows us to avoid finite dimensional approximations of the vector field. Results presented below are obtained by the facts that those cohomology groups satisfy axioms of the generalized cohomology theory (see below for a precise statement) and that they are invariant under the flow deformations. Those two facts were proved by A.Abbondandolo in \cite{abbo}. \\


\section{E-cohomology.}

$E$-cohomology groups are defined to be the direct limit of the certain ordinary cohomology groups. Let us remark on the latter first.  The main requirement is that the theory satisfies the strong excision axiom. There are various possible choices, e.g. $\check{C}$ech cohomology groups (\cite{GG}) or Alexander-Spanier cohomology groups (\cite{abbo}). We choose the homotopical point of view. Let us recall the definition from the Appendix in \cite{Crabb} (cf. \cite{Prieto} ch. 7). Denote by $K(\mathbb{F},n)$ an appropriate Eilenberg - Maclane space. For a topological pair ($X$,$Y$) define cohomology groups by
$$H^n(X,Y) = [X \cup \mbox{CY}, K(\mathbb{F},n)], $$
where $\mbox{CY}$ is a cone on $Y$. In the case when $X$ is a compact Hausdorff space and $Y$ is its closed subset  $H^n(X,Y)$ coincides with the Alexander - Spanier cohomology group and
$$H^n(X,Y) = [X/Y, K(\mathbb{F},n)].$$
One can also define cohomology groups with compact supports of a locally compact Hausdorff space $U$  by
$$H^n_c(U) = H^n(U^+,\ast) = [U, K(\mathbb{F},n)]_c,$$
where $U^+$ denotes the one-point compactification of $U$ and $[,]_c$ denotes homotopy classes of compactly supported maps.\\
Throughout rest of the paper, we take $\mathbb{F} = \mathbb{Z}_2$.\\

We are now ready to give an overview on what $E$-cohomology is. Let $E$ be a Hilbert space with a splitting $E = E^+ \oplus E^-$ where each of $E^+$ and $E^-$ is either infinite dimensional or trivial.
We say that $\{E_n\}_{n \in \mathbb{N}}$ is an approximating system for $E$ if
\begin{enumerate}
\item $E_n$ is a finite dimensional subspace of $E$ for every $n$;
\item there is an inclusion  $i_{n,n'}: E_n \hookrightarrow E_{n'}$ for every $n' > n$ ;
\item $\overline{\bigcup_n E_n }  = E$.
\end{enumerate}
We recall the definition of $E$-cohomology in two extremal cases: when $E^+ = \{0\}$, $E^- = l^2$ and when $E^+ = l^2$, $E^- = \{0\}$. For $l^2$ we take an approximation system induced by the spaces of finite sequences. However, one can prove (see \cite{GG}, \cite{abbo}) that the definition does not depend on the choice of the approximation system. \\ Let us first consider the case of $E =  \{0\} \oplus l^2  $. Take a closed and bounded set $X \subset E$. We define \textit{finite codimensional cohomology} in the following way (\cite{GD}). Put\\
 $E_n = \{(x_1,x_2,...) \in E: x_k = 0 \quad \mbox{for} \quad k > n\}$\\
$\hat{E}_n = \{(x_1,x_2,...) \in E_n: x_n \geq 0\}$\\
$\check{E}_n = \{(x_1,x_2,...) \in E_n: x_n \leq 0\}$\\
and
$X_n = X \cap E_n$, $\hat{X}_n = X \cap \hat{E}_n$, $\check{X}_n = X \cap \check{E}_n$. \\
\begin{center} \includegraphics[scale=0.7]{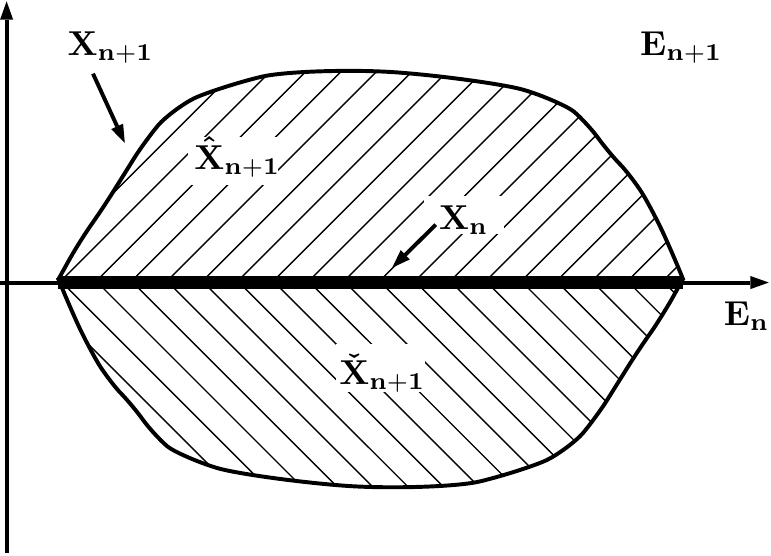} \end{center}

Since $\XX_{n+1} \cap \Xx_{n+1} = X_n$  and $\XX_{n+1} \cup \Xx_{n+1} = X_{n+1}$ the Mayer-Vietoris sequence for a triad $(X_{n+1},\XX_{n+1},\Xx_{n+1})$ gives a homomorphism
$$\delta_n: H^k(X_n) \to H^{k+1}(X_{n+1}). $$

\begin{defn}
Finite codimensional cohomology groups on $l^2$ are defined to be
$$H_E^{k}(X) = \varinjlim (H^{\dim E_n + k} (X_n) , \delta_n)$$
\end{defn}
Notice that if $k > 0$ then $\dim E_n + k > \dim X_n$. Thus, above groups can be nontrivial only for negative $k$. We have chosen a convention which is compatible with \cite{abbo} and opposite to that in \cite{GD}. This would be more convenient when we deal with the case when both $E^+$ and $E^-$ are nonzero.\\

As the simplest nontrivial example take $X = S(E)$ i.e. $X$ is a unit sphere in $E$. Then $X_n = S(E_n)$, $H^{n-1}(X_n) = \mathbb{Z}_2$ and all the maps $\delta_n$ are isomorphisms. Thus $H^{-1}_E(S(E)) \simeq \mathbb{Z}_2$ and $H^k_E(S(E))$ is trivial if $k \ne -1$ (notice that since $H^0(S(E_n))$ is mapped by $\delta_n$ to $H^{\dim{E_{n+1}} - \dim{E_n}}(S(E_{n+1}))$ we do not see zero'th cohomology of the sphere). Let us also emphasize that if  $X$ is compact (in particular if it is contained in a finite dimensional subspace) then all the $E$-cohomology groups are trivial.
For a general separable Hilbert space $E = \{0\} \oplus E^-$ we can take an isomorphism with $l^2$ (i.e. choose an approximating system) and repeat the construction.  \\
This simple concept can be generalized in many directions. K.G\c{e}ba and A.Granas (\cite{GG}) proved that for any generalized cohomology theory above groups are well defined. For example, taking cohomotopy groups instead of cohomology groups gives us \textit{stable} cohomotopy groups. In addition, they proved that the resulting theory is always a generalized cohomology theory on the Leray-Schauder category. Morphisms of the Leray - Schauder category are compact fields (i.e. $\mbox{Id} + K$ where $K$ is compact)so one could try to apply above techniques to fixed point theory. \\
However, it is good to extend the set of morphisms of the  Leray - Schauder category.  One reason for this is that one can not compare spheres of different radius via maps $\mbox{Id} + K$. Obviously, such spheres should have the same cohomology groups. \\
First of all, notice that above cohomology groups are trivially invariant under translations. A.Abbondandolo proved that they are also invariant under the flow deformations. This allows us to compare spheres of different radius and also, which is more important, to use Morse theory and the Conley Index techniques. \\
Another feature of \cite{abbo},\cite{KSZ} and \cite{sz} is a generalization to so-called \textit{middle dimension cohomology} i.e. the case when both $E^+$ and $E^-$ are of infinite dimension. Before introducing that, let us consider second extremal example: $E^+ = l^2$, $E^- = \{0\}$.  In this case $E$ - cohomology groups are defined by
$$H^k_E(X) = \varinjlim (H^k(X_m,\delta_m')).$$
where $\delta_m'$ is induced by the inclusion of $X_m$ into $X_{m+1}$. One can easily see that if $X$ is a sphere in $E_m$ then $H^k_E(X)$ is nontrivial (and equal to $\mathbb{Z}_2$) only if $k = m-1$ or $k = 0$. In fact, it is true that if $X$ is locally compact then $H^*_E$ is isomorphic to the \textbf{compactly supported} cohomology mentioned above. \\

 Now the \textit{middle dimensional cohomology groups} are defined to be
$$H_{E}^{k}(X) = \varinjlim (H^{\dim E^-_n + k} (X_{(m,n)} ,  \delta_m', \delta_n)),$$
where $X_{m,n} = X \cap (E_m \oplus E_n)$,  $\delta_n: X_{m,n} \to X_{m,n+1}$ is the map from the Mayer - Vietoris sequence and $\delta_m': X_{m,n} \to X_{m+1,n}$ is the map induced by inclusion. Again, this definition does not depend on the approximating system. \\
 One can think of $E$ cohomology as cohomology of finite codimension cohomology with respect to $E^-$ and cohomology with compact supports with respect to $E^+$. \\

We would like to emphasize the fact that $E$-cohomology groups satisfy axioms of a generalized cohomology theory (see Theorem $0.2$ in \cite{abbo}). All the results presented below can be obtained by these axioms without the knowledge of the precise construction of $E$-cohomology groups. For the sake of completeness, let us recall the homotopy invariance, the strong excision axiom and the long exact sequence for a triple. \\

\begin{defn}
A continuous map $\Psi: (X,A)  \to (Y,B) $ is an $E$-\textit{morphism} if:
\begin{enumerate}
  \item it has the form
  $$\Psi(x) = Lx + K(x),$$
            where $L$ is a linear automorphism of $E$ such that $LE^+ = E^+$ and $K$ maps bounded sets into precompact sets.
  \item $\Psi^{-1}(U)$ is bounded for every bounded set $U$.
\end{enumerate}
\end{defn}

We also say that $E$-morphisms $\Phi$ and $\Phi'$ from (X,A) to (Y,B) are $E$-homotopic if there exists an $E$-homotopy joining them i.e.
a continuous map $\Psi: (X,A) \times [0,1] \to (Y,B) $ such that
\begin{enumerate}
  \item
  $$\Psi(x,t) = L_tx + K(x,t),$$
            where $L_t$ is a linear automorphism of $E$ and $K$ maps bounded sets into precompact sets.
  \item $\Psi^{-1}(U)$ is bounded for every bounded set $U$
  \item $\Psi(\cdot,0) = \Phi$ and $\Psi(\cdot,1) = \Phi'$.
\end{enumerate}

Above definitions allow us to state:



\begin{itemize}
\item { (Homotopy invariance)  if two E - morphisms $\Phi$ and $\Phi'$ are $E$ - homotopic, then $H^*_E(\Phi) = H^*_E(\Phi')$};
\item (Strong excision) if $X$ and $Y$ are closed and bounded subsets of $E$ and $i\colon (X,X \cap Y) \to (X \cup Y, Y)$ is the inclusion map, then $H^*_E(i)$ is an isomorphism;
\item (Long exact sequence) For a triple $X \subset Y \subset Z$  of closed and bounded sets we have a long exact sequence
$$\ldots \to H^k_E(Z,Y) \to H^k_E(Z,X) \to H^k_E(Y,X)  \xrightarrow{\delta} H^{k+1}_E(Z,Y) \to H^{k+1}_E(Z,X) \to \ldots$$
\end{itemize}


In the proof of Proposition \ref{t:ipairs} we will also need two following lemmas (cf. \cite[pp. 372-373]{abbo} ).

\begin{lemma}
\label{lem1}
Let $Y$ be a closed subset of $X$. If there exists an $E$-homotopy
$$\Psi: (X,A) \times [0,1] \to (X,A) $$
such that $\Psi_0 = \mbox{Id}$, $\Psi_1(X) \subset Y$ and $\Psi_t(Y) \subset Y$ for every $t \in [0,1]$, then
$$H^*_E(X,A) \simeq H^*_E(Y,A),$$
the isomorphism being induced by the inclusion map.
\end{lemma}

\begin{lemma}
\label{lem2}
Let $B$ be a closed subset of $A$. If there exists an $E$-homotopy
$$\Psi: (X,A) \times [0,1] \to (X,A) $$
such that $\Psi_0 = \mbox{Id}$, $\Psi_1(A) \subset B$ and $\Psi_t(B) \subset B$ for every $t \in [0,1]$, then
$$H^*_E(X,A) \simeq H^*_E(X,B),$$
the isomorphism being induced by the inclusion map.
\end{lemma}

\section{Conley Index.}

We make the following assumptions throughout this paragraph.\\
Let $f \in C^2(E,\mathbb{R})$ be a function of the form
$$f(x) = \frac12 \langle Lx, x\rangle  + b(x), $$
where $L$ is a self-adjoint isomorphism, $\nabla b(x)$ is globally Lipschitz and $D^2b(x)$ is compact for every $x \in E$. Then operator $L$ gives a splitting of $E$ into  $E^+$ and $E^-$ corresponding to positive and negative eigenspaces respectively.
We would like to work with flows generated by the minus gradient equations i.e.
$$\dot{x} = -\nabla f (x).$$

We define the cohomological Conley Index in a Hilbert space $E$ to be the $E$-cohomology of an index pair in $E$. This is a different approach than in \cite{Izy} and \cite{Man} because it does not use finite dimensional approximations of the vector field. We compare our approach to \cite{Izy} after proving Proposition \ref{t:ipairs}. \\
Following \cite[ p.221]{GIP} we define an isolating neighborhood in a Hilbert space.
\begin{defn}
We say that a bounded and closed set $N \subset E$ is an isolating neighborhood if
$$\mbox{Inv}(N) \subset \mbox{int} N$$
\end{defn}

\begin{defn}
\label{d:ipair}
Let $N$ be an isolating neighborhood of an invariant set $S$. We call a closed and bounded pair $(N_1,N_0)$ an \emph{index pair} for $S$ if
\begin{enumerate}
\item $N_0$ is positively invariant relative to $N_1$,
\item $S \subset \mbox{int} N_1 \setminus N_0$ and
\item if $\gamma \in N_1$, $t > 0$ and $\gamma \cdot t \not\in N$, then there exists $t'$ such that $\gamma \cdot [0,t'] \subset N_1$ and $\gamma \cdot t' \in N_0$.

\end{enumerate}
Moreover, we say that an index pair is regular if the function
$$\tau(x) = \inf \{s \in \mathbb{R}^{\geq 0}: x \cdot [0,s] \not \subset N_1 \setminus N_0  \}$$
is continuous.
\end{defn}
\quad \\

Unless otherwise stated, we assume that all index pairs are regular (c.f. Remark \ref{exist} for the existence).\\

We say that an index pair $(N_1,N_0)$ is \textit{contained} in the isolating neighborhood $N$ if $N_0 \subset N_1 \subset N$ and $N_1$, $N_0$ are positively invariant relative to $N$ (c.f. \cite{smol} p.489).

\begin{defn}
We define the \textbf{cohomological Conley Index} of $S$ (denoted by $\mbox{ch}^*(S)$) to be $H^*_E(N_1,N_0)$, where $(N_1,N_0)$ is an index pair for $S$.
\end{defn}

Above definition only makes sense if we prove the independence of a choice of index pairs. This is stated in the following proposition.


\begin{prop}
\label{t:ipairs}
Suppose we have a flow as in the beginning of this section. Let $(N_1,N_0)$, $(\hat{N}_1, \hat{N}_0)$ be two regular index pairs for $S = \mbox{Inv}(N)$ contained in the same isolating neighborhood $N$. Then
$$H^*_E(N_1,N_0) \simeq H^*_E(\hat{N}_1,\hat{N}_0).$$
\end{prop}
\quad

Notice that the assumption on the flow is crucial.

Define the set $N_1^t$, $N_0^{-t}$ by
$$  N_1^t = \{ x \in N_1 : x \cdot [-t,0] \subset N_1\} $$
$$    N_0^{-t} = \{ x \in N_1: \exists_{y \in N_0}  \exists_{t' \in [0,t]} \quad y \cdot [-t',0] \subset N_1, y \cdot (-t') = x\}.$$

\begin{center} \includegraphics[scale=1.0]{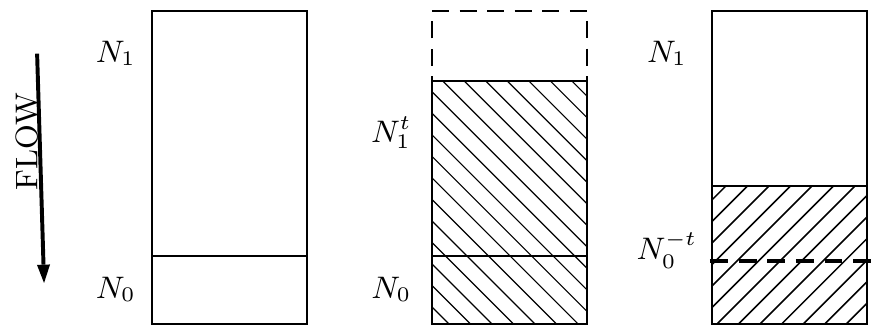} \end{center}

The proof of  Proposition \ref{t:ipairs} can be divided into three steps. \\

\textit{step 1:} For every $t > 0$ there is an isomorphism
 $$H^*_E(N_1,N_0) \to H^*_E(N_1^t,N_0 \cap N_1^t) $$

\textit{step 2:} For every $t > 0$ there is an isomorphism
 $$H^*_E(N_1,N_0) \to H^*_E(N_1,N_0^{-t})$$

\textit{step 3:} There exists $T > 0$ such that
$$(N_1^T,N_0 \cap N_1^T) \subset (\hat{N}_1, \hat{N}_0^{-T})$$
 $$(\hat{N}_1^T,\hat{N}_0 \cap \hat{N}_1^T) \subset (N_1, N_0^{-T})$$
and the inclusions induce isomorphisms of E-cohomology groups.

\begin{proof} \quad \\
Step 1: \\
Define $\Psi: (N_1,N_0) \times [0,1] \to (N_1,N_0)$ by
$$\Psi(x,s) = \left\{
             \begin{array}{ll}
               x \cdot s, & \hbox{if $x \cdot [0,s] \in N_1 \setminus N_0$;} \\
               x \cdot \tau(x), & \hbox{otherwise.}
             \end{array}
           \right.
$$

Put $X = N_1$ , $Y = N_1^t \cup N_0$, $A = N_0$. Lemma (\ref{lem1}) gives us
$$H^*_E(N_1,N_0) \simeq H^*_E(N_1^t \cup N_0, N_0) $$
From the excision axiom we have
$$H^*_E(N_1^t,N_0 \cap N_1^t) \simeq H^*_E(N_1^t \cup N_0, N_0).$$

Step 2 can be done in a similar way as Step 1. \\

Step 3: \\
Take $\hat{N} = \mbox{cl} (N \setminus N_1)$. For every $x \in \hat{N}$ there exists $T_x > 0$ such that $x \cdot (-T_x) \not\in \hat{N}$. In fact, we will show that there exists $T_1$ which satisfies above condition for every $x \in \hat{N}$. In a finite dimensional case this is just a consequence of the compactness of $\hat{N}$.\\
Suppose we have a sequence $(y_n) \subset \hat{N}$ such that $y_n \cdot (-2n,0) \subset \hat{N}$. Put $x_n := y_n \cdot (-n)$. Then both sets $x_n \cdot (-n,0)$ and $x_n \cdot (0,n)$ are contained in $\hat{N}$.
If $x_{n_k} \to x_0 $ then $x_0 \cdot (-\infty,0) \subset \hat{N}$ and $x_0 \in \hat{N}$ and we have arrived at a contradiction.

We now prove that if $x_n \cdot (-n,n) \subset \hat{N}$ for every $n$ then $(x_n)$ contains a convergent subsequence. \\

Suppose that $(x_n^+) \subset E^+$ does not have a convergent subsequence. Then there exists $\epsilon > 0$ such that $|x_n^+ - x_m^+| > \epsilon$ for every $n \neq m$.\\
Take $s, T_1 > 0$ such that
$$N \subset B(0,s)$$
$$|e^{T_1L}x| > \frac{3s}{\epsilon} |x|$$
for every $x \in E^+$. Then for $n,m > T_1$ we have

$$3s  < |e^{T_1L}(x_n - x_m)| \leqslant |x_m \cdot T_1| + |x_n \cdot T_1| + |K(x_m,T_1) - K(x_n,T_1)| \leqslant$$
$$ \leqslant 2s + |K(x_m,T_1) - K(x_n,T_1)|$$

and so $|K(x_m,T_1) - K(x_n,T_1)|  > s$ for every $n,m > T_1$. However, $K(\cdot,T_1)$ is compact and we have a contradiction. As a consequence, we can choose a convergent subsequence $(x_{n_l}^+)$. In a similar way, from  $(x_{n_l}^-)$ we can take a convergent subsequence $(x_{n_k}^-)$ and this gives us a convergence of $(x_{n_k})$.\\

By the same argument we can find $T_2 > 0$ such that for every $x \in N_0$ we have $x \cdot T \not\in N_0$. Take $T = \max \{ T_1,T_2, \overline{T_1}, \overline{T_2}\}$ where $\overline{T_1}, \overline{T_2}$ correspond to the pair $(\hat{N_1}, \hat{N_0})$. Then
$$(N_1^T,N_0 \cap N_1^T) \subset (\hat{N}_1, \hat{N}_0^{-T})$$
 $$(\hat{N}_1^T,\hat{N}_0 \cap \hat{N}_1^T) \subset (N_1, N_0^{-T})$$
 The proof that above inclusions induce isomorphisms on the cohomology groups runs as in the finite dimensional case (see \cite{smol} p.486-492 or \cite{Chang}, p.401 ).
\end{proof}

\begin{rem}
One can show an independence of the Conley Index without the assumption that both index pairs are contained in the same isolating neighborhood (c.f. \cite[ p.491 ]{smol}).

\end{rem}

\begin{rem}\label{exist}
We want to emphasize that for an isolated invariant set $S$ there exists a regular index pair. Let $U$ be an isolating neighbourhood and define
$G^T(U) = \bigcap_{|t| \leqslant T} U \cdot (-t,t)$. Then there exists $T > 0$ such that $G^T(U) \subset \inter{U}$. \\
Suppose the converse, i.e. $G^n(U) \not\subset \inter{U}$ for every $n$. Take $x_n \in G^n(U) \setminus \inter{U}$ i.e. $x_n \cdot (-n.n) \subset U$. There exists a convergent subsequence $x_{n_k} \to x_0 \in S \subset \inter{U}$ (see the proof above). A contradiction.\\
This proves that $G^T(U) \subset \inter{U}$ for some $T > 0$. For such an $U$ one can construct a regular index pair (see Theorem 5.5.13 in \cite{Chang}).
\end{rem}

Let us compare our definition of the cohomological Conley to the one which uses finite dimensional approximations of the vector field (\cite{Izy}). For a compact $K$ define $K_n : E \to E$ by
$$K(x) = P_n \circ K \circ P_n (x).$$
Let $S$ be an isolated invariant set for the flow generated by $F = L + K$ and let $\hat{N}$ be an isolating neighborhood for $S$. Then $F$ is related by  continuation (through E-homotopies) to $F_{n_0} = L + K_{n_0}$ for sufficiently large $n_0$.  Let $(N,L)$ be an index pair for the approximation i.e. for the finite dimensional flow generated by ${F_{n_0}}_{|E_{n_0}} : E_{n_0} \to E_{n_0}$. Clearly $(N_E,L_E) = (N \times D^{+}_{-n} \times D^{-}_{-n} , L \times  D^+_{-n} \times D^-_{-n} \cup N \times  D^+_{-n} \times \partial D^-_{-n} )$ is an index pair for $F$ in $E$, where $D^{+/-}_{-n}$ denotes a  disc in $E_n^\perp \cap E^{+/-}$. It is easy to check that $H^*_E(N_E,L_E)$ coincides with the cohomological Conley index defined in \cite{Izy}. Since $E$-cohomology does not depend on the index pair, those two approaches coincide in general.


\section{Main Theorem.}
Let us recall that we are interested  in a flow generated by the minus gradient vector field for a function
$f \in C^2(E,\mathbb{R})$ of the form
$$f(x) = \frac12 \langle Lx, x\rangle  + b(x), $$
where $L$ is a self-adjoint isomorphism, $\nabla b(x)$ is globally Lipschitz and $D^2b(x)$ is compact for every $x \in E$. \\ 

Let $S$ be a compact isolated invariant set containing only non-degenerate critical points $x_1$,...,$x_n$ and orbits connecting them.    \\
For a non-degenerate critical points $x$ we define an $E$-index by
$$\eind x = \dim V \cap E^+ - \dim V^\perp \cap E^-  = \dim V \cap E^+ - \codim E^+ +  V $$
where $V$ is the negative eigenspace of $D^2f(x)$, $E^+$ and $E^-$ are respectively positive  and negative  eigenspaces of $L$.\\
Suppose further that the transversality condition holds i.e. if $\eind y - \eind x = 1$ the stable manifold of $y$ and unstable of $x$ intersect transversally.

    \begin{thma}
\label{main}
We have
$$HF^*(S) \simeq  \mbox{ch}^*(S),$$
where $HF^*(S)$ denotes Floer cohomology.
\end{thma}

D.Salamon used an analogous theorem in a following way (see \cite{sal}). Take a function $f$ on finite dimensional closed manifold $M$. Then $(M, \emptyset)$ is an index pair for the isolated invariant set $S = M$. If the Morse cohomology is isomorphic to the cohomological Conley Index, we have
$$H^*(M) = H^*(M,\emptyset) \simeq H_{\mbox{Morse}}^*(M)$$

Thus, this is just another proof that  Morse theory recovers singular cohomology groups. \\

Let us first give the main ideas of the proof. Take two non-degenerate critical points $x$ and $y$ of a relative index $1$ and a connecting orbit $C$. By the transversality, $\hat{S} = \{x,y,C\}$ is an isolated invariant set. Now choose a triple $(N_2,N_1,N_0)$ in such a way that the pairs $(N_2,N_0)$, $(N_2,N_1)$, $(N_1,N_0)$ are index pairs for the invariant sets $\hat{S}$, $\{y\}$ and $\{x\}$ respectively (given a pair $(N_2,N_0)$ put $N_1 = N_2 \cap f^{-1}((-\infty,b])$ for an appropriate $b$). \\
\begin{center} \includegraphics[scale=0.7]{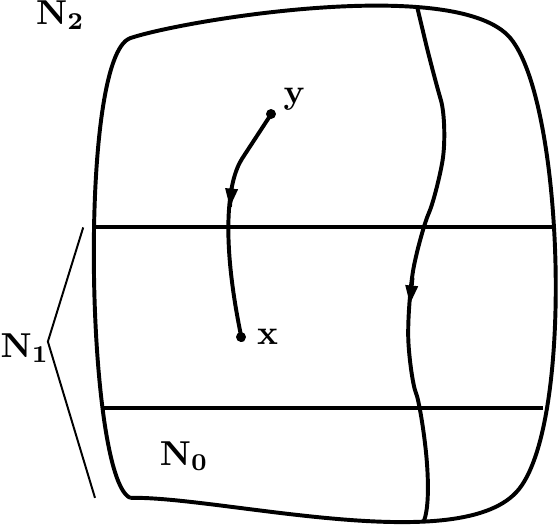} \end{center}
We have the long exact sequence
\begin{equation}\label{exact}\ldots \to H^{k}_E(N_2,N_0) \to H^{k}_E(N_1,N_0) \to H^{k+1}_E(N_2,N_1) \to H^{k+1}_E(N_2,N_0) \to
\end{equation}
One can show that the cohomological Conley Index for $S$ is trivial, i.e. all the groups $H^{n}_E(N_2,N_0)$ are trivial. Thus, for every $k$ we have an isomorphism
$$H^{k+1}_E(N_2,N_1) \to H^{k}_E(N_1,N_0)$$
A.Abbondandolo computed (see Proposition 14.6 in \cite{abbo}) the cohomological Conley Index for a non-degenerate critical point.

$$\mbox{ch}^k(\{x\}) = \left\{
\begin{array}{cc}
\mathbb{Z}_2 \ \mbox{for} \ k = \eind x\\
0 \quad \mbox{otherwise}
\end{array}
\right. $$
 The only nontrivial morphism in (\ref{exact}) is an isomorphism between $H^{k+1}_E(N_2,N_1)$ and $H^{k}_E(N_1,N_0)$ where $k = \eind x$. \\
By the compactness of $S$ and the transversality condition, we have finitely many orbits $C_1$, $C_2$,  $\ldots$, $C_m$ connecting $y$ and $x$. Take $\hat{S} = \{x,y, C_1, \ldots, C_m\}$.  By the additivity (see \cite[p. 201]{mcc}), the Conley connection matrix is a sum of isomorphisms from $\mathbb{Z}_2$ to itself so it is an \textbf{ algebraic count modulo 2}. This is exactly the \textbf{Floer boundary operator}. Since $ch^*(\{x\}) \simeq \mathbb{Z}_2$ one can think of $ch^*(\{x\})$ as of a generator of the Floer chain group $C_{\eind x}$.  \\

Here are some technical details of the above construction.  \\
We would like to prove that the Conley Index of $\hat{S} = \{x,y,C\}$ is trivial. Let us examine a special case. Suppose $C'$ is contained in one dimensional subspace $E_1$ and $x' = (-1,0) \in E_1 \oplus E_1^\perp$, $y' = (1,0)$, $C' = [-1,1] \times \{0\}$.
\begin{lemma}
\label{special}
 The Conley Index of $\hat{S}$ is trivial.
\end{lemma}
\begin{proof}
We follow an approach of C.McCord (\cite{mcc}) i.e. we use a series of continuations. Choose a small isolating neighborhood $N$ of $S'$.  First continue the vector field $F(x,y) = (F_x(x,y), F_y(x,y))$ to $F_1(x,y) = (F_x(x,0) + D_yF_x(x,0)y, F_y(x,0) + D_yF(x,0)y) = (F_x(x,0) + D_yF_x(x,0)y,D_yF(x,0)y)$ and then to $F_2(x,y) = (F_x(x,0), D_yF(x,0)y)$. Now put $a(x) = F_x(x,0)$, $M = \max_{x \in [0,1]} a(x)$ and continue $F_2(x,y)$ to
$$F_3(x,y) = (F_x(x,0) - M - 1, D_yF(x,0)y)  .$$
Notice that $\mbox{inv}{(F_3, S)} = \emptyset$ and thus the Conley Index is trivial.
\end{proof}

Now we would like to find an $E$-homotopy which  reduces a general case to the above one (compare section C in \cite{proper}).\\

Let $M$ be a  compact $C^1$ submanifold of a Hilbert space $E$.

\begin{lemma}
\label{projecton}
There exists a finite dimensional subspace $T$ of $E$ such that the orthogonal projection $P_T$  onto $T$ maps $M$ diffeomorphically onto $P_T(M)$.
\end{lemma}
 \begin{proof}
 For every $x \in M$ there is an open neighbourhood $U_x$ such that $U_x$ is diffeomorphic to the open neighbourhood of $0$ in $T_xM$ via the exponential map. Choose a finite subcover $U_{x_1}$,$U_{x_2}$, ..., $U_{x_k}$ and put
 $$T' = \mbox{span}\{T_{x_i}M: i = 1, \ldots, k \}.$$
The orthogonal projection $P_{T'|M}: M \to T'$ is an imbedding. Thus, for a given $x \in M$, there is only a finite number of points $y_1$,$y_2$, ..., $y_p$ such that $P_{T'}x = P_{T'}y_i$. Define $T'_x$ to be the space spanned by $T'$ and $y_1 - x$, $y_2 - x$, ... , $y_p-x$ and let $P_x$ be the orthogonal projection onto $T'_x$. It is easy to see that there is an open neighbourhood $V_x$ of $x$ such that $y \in V_x$, $z \in M$ and $P_xz = P_xy$ imply $z = y$. Again, choose a finite cover $V_{x_1}$,$V_{x_2}$,...,$V_{x_q}$ and put
$$T = \mbox{span}\{T'_{i}: i = 1, \ldots, q \}$$
 \end{proof}

 \begin{lemma}
There exists an $E$-homotopy and a finite dimensional subspace $T_1$
\begin{enumerate}
\item $\Psi(\cdot,0) = \mbox{Id}$,
\item $\Psi(M,1)$ is contained in $T_1$.
\end{enumerate}
\end{lemma}

 \begin{proof}
 By Lemma \ref{projecton} we can find a finite dimensional space such that $P_T:M \to T$ is an injection. For $x \in M$  define $\phi_x:M \to E$ by
 $$\phi_x(y) = P_T(y-x) + x$$
 Then $\phi_x$ is an imbedding and $\phi_x(x) = x$. Let $U_{x_1}$, ..., $U_{x_k}$ be a cover of $M$ and $\{\nu_i\}$ be a subordinated partition of unity. Define $\phi: M \to E$ by $\phi(x) = \sum \nu_i(x) \phi_{x_i}(x)$. Take $T_1$ to be the space spanned by $T$ and $x_1$, ..., $x_k$. Then $\phi(M) \subset T_1$. \\
 Define $\eta_0: P_T(M) \to T^\perp$ by $\eta(P_Tx) = x - \phi(x)$. Since $P_T(M)$ is a $C^1$ submanifold of $T$ we can extend $\eta$ to a $C^1$ map on $T$. Define $\Psi(x,y) = x - t\eta(P_Tx)$.
 \end{proof}

Two critical points together with  an orbit between them is a compact submanifold of $E$. Thus we can apply above lemma to $M = \hat{S}$. Suppose $\Phi(\hat{S},1)$ is contained in a finite dimensional space $T_1$. Choose a one dimensional subspace $E_1 \subset T_1$ and a diffeomorphism $h$ of $T_1$ which takes $\hat{S}$ onto $(-1,1) \subset E_1$. Extend $h$ to $E$ by the identity on $T^{\perp}$. This reduces a general case to the one in Lemma \ref{special}.


\section{Further Directions.}
Some of the Floer theories  come with additional symmetry. One expects an analogous theorem to the main theorem of this paper for the equivariant Morse cohomology and the equivariant Conley Index in a Hilbert space. \\
For an $S^1$-action  there is  the conjecture that the Monopole Floer cohomology and the Seiberg-Witten Floer cohomology are isomorphic ($\mbox{HM}^*(Y) \simeq \mbox{HSW}^*(Y)$). However, this should be treated more carefully since one cannot assume the existence of the (local) flow in a Hilbert space. \\
Another direction, that one would like to investigate, is the case of Hilbert (Banach) manifolds. Let us just recall that recently intensively explored Lagrangian intersection Floer theory (see \cite{fukaya}) is a Floer theory on a Banach manifold.

 \subsection*{Acknowledgement} I am indebted to K.G\k{e}ba and M.Izydorek, who provided valuable comments on an earlier version. I would also like to thank J.Maksymiuk for pointing out results in \cite{Chang} and to the referee for a number of helpful suggestions for improvement in the article.

\newpage

\vspace{1cm}
Maciej Starostka\\
Institute of Mathematics\\
Polish Academy of Sciences\\
00-956 Warsaw\\
ul. Sniadeckich 8\\
and\\
Gdansk University Of Technology\\
80-233 Gdañsk,\\
ul. Gabriela Narutowicza 11/12 \\
Poland\\
E-mail: maciejstarostka@gmail.com

\end{document}